\documentclass[10pt, reqno]{amsart}
\usepackage{amsmath, amsthm, amscd, amsfonts, amssymb, graphicx, color}
\usepackage[bookmarksnumbered, colorlinks, plainpages]{hyperref}
\hypersetup{colorlinks=true,linkcolor=red, anchorcolor=green, citecolor=cyan, urlcolor=red, filecolor=magenta, pdftoolbar=true}
\usepackage[cp1251]{inputenc}
\newtheorem{theorem}{Theorem}[section]
\newtheorem{lemma}[theorem]{Lemma}

\newtheorem{corollary}[theorem]{Corollary}
\theoremstyle{definition}
\newtheorem{definition}[theorem]{Definition}

\theoremstyle{remark}
\newtheorem{remark}[theorem]{Remark}
\numberwithin{equation}{section}

\begin{document}

\setcounter{page}{1}

\title[On inhomogeneous exterior Robin problems ...]{On inhomogeneous exterior Robin problems with critical nonlinearities}

\author[M. B. Borikhanov and B. T. Torebek]{Meiirkhan B. Borikhanov and  Berikbol T. Torebek}

\address{\textcolor[rgb]{0.00,0.00,0.84}{Meiirkhan B. Borikhanov \newline Khoja Akhmet Yassawi International Kazakh--Turkish University \newline Sattarkhanov ave., 29, 161200 Turkistan, Kazakhstan  \newline Institute of
Mathematics and Mathematical Modeling\newline 125 Pushkin str., 050010 Almaty, Kazakhstan}}
\email{\textcolor[rgb]{0.00,0.00,0.84}meiirkhan.borikhanov@ayu.edu.kz}

\address{\textcolor[rgb]{0.00,0.00,0.84}{Berikbol T. Torebek \newline Institute of
Mathematics and Mathematical Modeling \newline 125 Pushkin str.,
050010 Almaty, Kazakhstan \newline Department of Mathematics: Analysis, Logic and Discrete Mathematics \newline
Ghent University, Belgium}}
\email{\textcolor[rgb]{0.00,0.00,0.84}{berikbol.torebek@ugent.be}}

%\thanks{}

\let\thefootnote\relax\footnote{$^{*}$Corresponding author}

\subjclass[2020]{35K70, 35A01, 35B44.}

\keywords{exterior Robin problem, critical exponent, nonexistence, global solution.}

\begin{abstract}
The paper studies the large-time behavior of solutions to the Robin problem for PDEs with critical nonlinearities. For the considered problems, nonexistence results are obtained, which complements the interesting recent results by Ikeda et al. [J. Differential Equations, 269 (2020), no. 1, 563-594], where critical cases were left open. Moreover, our results provide partially answers to some other open questions previously posed by Zhang [Proc. Roy. Soc. Edinburgh Sect. A, 131 (2001), no. 2, 451-475] and Jleli-Samet [Nonlinear Anal., 178 (2019), 348-365].
\end{abstract}
\maketitle
\tableofcontents
\section{Introduction}
Recently, Ikeda et al. in \cite{Ikeda} have studied the exterior inhomogeneous Robin problems for the following semilinear PDEs:

$\bullet$ elliptic equation
\begin{equation}\label{P1}\left\{\begin{array}{l}
-\Delta u={{u}^{p}},\,\,x\in\Omega^c, \\{}\\
\large\displaystyle\frac{\partial u}{\partial \nu}+\alpha u=f(x),\,\,\,x\in \partial\Omega,
\end{array}\right.\end{equation}

$\bullet$ heat equation
\begin{equation}\label{P2}\left\{\begin{array}{l}
u_t-\Delta u={{|u|}^{p}},\,\,(t,x)\in (0,\infty)\times\Omega^c, \\{}\\
\large\displaystyle\frac{\partial u}{\partial \nu}+\alpha u=f(x),\,\,(t,x)\in (0,\infty)\times\partial\Omega,\\{}\\
u\left( 0,x \right)=u_0\left( x \right),\,\,x\in \Omega^c,
\end{array}\right.\end{equation}

$\bullet$ wave equation
\begin{equation}\label{P3}\left\{\begin{array}{l}
u_{tt}-\Delta u={{|u|}^{p}},\,\,(t,x)\in (0,\infty)\times\Omega^c, \\{}\\
\large\displaystyle\frac{\partial u}{\partial \nu}+\alpha u=f(x),\,\,(t,x)\in (0,\infty)\times\partial\Omega,\\{}\\
u\left( 0,x \right)=u_0\left( x \right),\,\,\,
u_t\left( 0,x \right)=u_1\left( x \right),\,\,x\in \Omega^c,
\end{array}\right.\end{equation}

$\bullet$ damped wave equation
\begin{equation}\label{P4}\left\{\begin{array}{l}
u_{tt}-\Delta u+u_t={{|u|}^{p}},\,\,(t,x)\in (0,\infty)\times\Omega^c, \\{}\\
\large\displaystyle\frac{\partial u}{\partial \nu}+\alpha u=f(x),\,\,(t,x)\in (0,\infty)\times\partial\Omega,\\{}\\
u\left( 0,x \right)=u_0\left( x \right),\,\,
u_t\left( 0,x \right)=u_1\left( x \right),\,\,x\in \Omega^c,
\end{array}\right.\end{equation}
where $\Omega=B(0,1)$ is the unit open ball in $\mathbb{R}^N$, $N\geq3, \Omega^c=\mathbb{R}^N\setminus\Omega,\,\alpha>0,\,p\in(1,\infty), f\in L^1(\partial\Omega),\,f\not\equiv0$ and $u_0,u_1\in L^1_{\text{loc}}(\overline{\Omega^c})$. In addition, $\nu$ is the outward unit normal on $\partial\Omega$.

They showed that:

{\it Suppose that $u_i \in L^1_{loc}(\Omega^c), u_i \geq 0, i = 0, 1,$ and $f \in L^1(\partial\Omega).$}
\begin{description}
\item[(a)] If $$1 < p < \frac{N}{N -2}, N\geq 2$$ and $$\int_{\partial\Omega} f (x) d\sigma > 0,$$ then problems \eqref{P2}, \eqref{P3} and \eqref{P4} admit no global weak solutions. In particular, problem \eqref{P1} has no positive solutions.
\item[(b)] If $$p > \frac{N}{N -2},\,N\geq 3,$$ the problems \eqref{P2}, \eqref{P3} and \eqref{P4} has positive global solutions for some $f > 0.$ In particular, problem \eqref{P1} has positive solution for some $f > 0.$
\end{description}

They also studied the dependence of the critical exponent on the initial data and obtained the following results:

{\it Suppose that $u_i \in L^1_{loc}(\Omega^c), u_i \geq 0, i = 0, 1,$ and $f \in L^1(\partial\Omega),\,f\geq 0.$ Let $u_0$ satisfies
\begin{equation}\label{u0}
u_0(x)\geq A|x|^\lambda,\,\,\,\,A>0,\,\,\,\,2-N<\lambda<0.
\end{equation}}
\begin{description}
\item[(a*)] If $$1 < p < 1-\frac{2}{\lambda},$$ then problems \eqref{P2}, \eqref{P3} and \eqref{P4} possess no global weak solutions.
\item[(b*)] If $$p > 1-\frac{2}{\lambda},$$ then problems \eqref{P2}, \eqref{P3} and \eqref{P4} have global positive solutions for some $f > 0.$
\end{description}

It should be noted that in the above results the critical cases $p=\frac{N}{N -2},$ and $p = 1-\frac{2}{\lambda}$ were not investigated and the large-time behavior of the solution in the critical cases is still {\bf open}. The main purpose of our paper is to attempt to answer this question and fill the gap in the results of Ikeda et al. from \cite{Ikeda}.

For the convenience of readers, below we will provide more detailed historical facts about the above problems and about critical exponents.

When $\Omega$ is empty, Fujita (see \cite{Fujita}) showed that the critical exponent of the parabolic problem \eqref{P2} is $1+\frac{2}{N},$ i.e.:
\begin{itemize}
\item[(i)] If $1 <p< 1+\frac{2}{N}$ and $u_0\geq 0,$ then problem \eqref{P2} has no global positive solutions;
\item[(ii)] If $p > 1+\frac{2}{N}$ and $u_0$ is  smaller than a small Gaussian, then \eqref{P2} has global positive solutions.
\end{itemize}
The number $1+\frac{2}{N}$ is called the Fujita critical exponent. When $\Omega$ is bounded non-empty domain, Bandle and Levine \cite{Bandle} showed that the critical exponent of semilinear heat equation with homogeneous exterior Dirichlet conditions is still $1+\frac{2}{N}$. Also, in \cite{Levine, Rault} it was shown that the number $1+\frac{2}{N}$ is still a critical exponent for the semilinear heat equation with homogeneous exterior Neumann and with homogeneous exterior Robin conditions. The number $1+\frac{2}{N}$ is also a critical exponent for the exterior damping wave problems (see \cite{Dabbicco, Fino, Ikeda1, Ikehata, Ogawa}).

In \cite{Bandle2} Bandle et al. studied the parabolic problem \eqref{P2} subject to the inhomogeneous Dirichlet and Neumann boundary conditions. Namely, it was shown that
\begin{itemize}
\item[(i)] If $1<p\leq \frac{N}{N-2},\,N\geq 2$ and $\int_{\partial\Omega} f(x) d\sigma>0,$ then \eqref{P2} has no global solutions for $\alpha=0$ or $\alpha=\infty.$
\item[(ii)] If $p> \frac{N}{N-2},\,N\geq 3,$ then problem \eqref{P2} for $\alpha=0$ or $\alpha=\infty,$ has global solutions for some $f>0$ and $u_0>0.$
\end{itemize}
In \cite{Zhang} Zhang generalized these results considering weighted nonlinearity $|x|^m|u|^p,\,m>-2$ instead of the nonlinear term $|u|^p$ and showed that $\frac{N+m}{N-2}$ is a critical exponent. In addition, he also studied semilinear elliptic equation (with nonlinearity $|x|^m|u|^p,\,m>-2$) with exterior Dirichlet or exterior Neumann non-zero boundary conditions and semilinear wave equations (with nonlinearity $|x|^m|u|^p,\,m>-2$) with exterior Neumann inhomogeneous boundary conditions. It is shown that $\frac{N+m}{N-2}$ is still a critical exponent of these problems. However, critical cases of these problems corresponding to the Neumann boundary conditions were not studied and remained open.

It is worth noting that several extensions of the above results to more general types of parabolic and hyperbolic equations with exterior Dirichlet or Neumann boundary conditions have been studied in a number of papers (for example, see \cite{Jleli2, Jleli0, Jleli1, Laptev1, Pinsky, Sun, Zeng, Zeng1}).

Recently, Jleli and Samet \cite{Jleli} obtained a number of interesting results for the heat and wave (with and without damping) equations with inhomogeneous Dirichlet conditions on exterior domains. That is, they showed that the critical exponent of the studied problems is $\frac{N}{N-2},$ they also proved that in the two-dimensional case, the studied problems do not have global positive solutions for all $p>1$. However, the critical cases $p=\frac{N}{N-2}$ of the studied problems have also not been studied and remain open \cite[Remark 1.7]{Jleli}.

\section{Main results}
In this section, we derive the main results of the present paper.

Before presenting our main results, let us mention in which sense the solutions are considered.

Let $Q:=(0,\infty)\times\Omega^c$ and $\Gamma:=(0,\infty)\times\partial\Omega$, here we have to note that $\Omega^c$ is closed and $\Gamma\subset \Omega^c.$
\begin{definition}[Weak solution]\label{WS1} Let $f\in L^1(\partial\Omega)$. We say that $u\in L^p_\text{loc}(\overline{\Omega^c})$ is a positive weak solution to \eqref{P1}, if
\begin{equation}\label{WP1}\begin{split}
&\int_{\Omega^c}u^{p}\varphi_2  dx +\int_\Gamma f\varphi_2 d\sigma = -\int_{\Omega^c} u\Delta\varphi_2 dx ,
\end{split}\end{equation}
holds for all $\varphi_2\in {C^{2}(\overline{\Omega^c})}, \varphi\geq0$ such that
\begin{description}
\item[(i)] $\varphi(x)\equiv0,\,\,|x|\geq R$ here $R>1$;
\item[(ii)] $\large\displaystyle\frac{\partial }{\partial \nu}\varphi_2(x)+\alpha \varphi_2(x)=0,\,\,x\in \partial\Omega$;
\item[(iii)] $\Delta \varphi_2\in C(\overline{\Omega^c})$
\end{description}
and the notation $d\sigma$ is the surface measure on $\partial\Omega$.
\end{definition}

\begin{definition}[Global weak solution]\label{WS2} Let $u_0\in L^1_{\text{loc}}(\overline{\Omega^c})$ and $f\in L^1(\partial\Omega)$. We say that $u\in L^p_\text{loc}([0,\infty)\times\overline{\Omega^c})$ is a global weak solution to \eqref{P2}, if
\begin{equation}\label{WP2}\begin{split}
&\int_Q|u|^{p}\varphi  dx dt +\int_{\Omega^c} u_0(x)\varphi(0,x) dx+\int_\Gamma f\varphi d\sigma dt
= -\int_Q u\varphi_tdx dt-\int_Q u\Delta\varphi dx dt,
\end{split}\end{equation}
holds for all $\varphi\in C_{t,x}^{1,2}((0,\infty)\times\Omega^c)\cap([0,\infty)\times\overline{\Omega^c}), \varphi\geq0$ such that
\begin{description}
\item[(i)] $\varphi(t,\cdot)\equiv0,\,\,t\geq T$ here $T>0$;
\item[(ii)] $\varphi(\cdot,x)\equiv0,\,\,|x|\geq R$ here $R>1$;
\item[(iii)] $\large\displaystyle\frac{\partial }{\partial \nu}\varphi(x)+\alpha \varphi(x)=0,\,\,x\in \Gamma$;
\item[(iv)] $\varphi_t, \Delta \varphi\in C([0,\infty)\times\overline{\Omega^c})$.
\end{description}
\end{definition}

\begin{definition}[Global weak solution]\label{WS3} Let $u_0, u_1\in L^1_{\text{loc}}(\overline{\Omega^c})$ and $f\in L^1(\partial\Omega)$. We say that $u\in L^p_\text{loc}([0,\infty)\times\overline{\Omega^c})$  is a global weak solution to \eqref{P3}, if
\begin{equation}\label{WP3}\begin{split}
\int_Q|u|^{p}\varphi  dx dt +\int_{\Omega^c} u_1(x)\varphi(0,x) dx&-\int_{\Omega^c} u_0(x)\varphi_t(0,x) dx+\int_\Gamma f\varphi d\sigma dt
\\&= -\int_Q u\varphi_{tt}dx dt-\int_Q u\Delta\varphi dx dt,
\end{split}\end{equation}
holds for all $\varphi\in C_{t,x}^{2,2}((0,\infty)\times\Omega^c)\cap([0,\infty)\times\overline{\Omega^c}), \varphi\geq0$ such that
\begin{description}
\item[(i)] $\varphi(t,\cdot)\equiv0,\,\,t\geq T$ here $T>0$;
\item[(ii)] $\varphi(\cdot,x)\equiv0,\,\,|x|\geq R$ here $R>1$;
\item[(iii)] $\large\displaystyle\frac{\partial }{\partial \nu}\varphi(t,x)+\alpha \varphi(t,x)=0,\,\,(t,x)\in\Gamma$;
\item[(iv)] $\varphi_{tt}, \Delta \varphi\in C([0,\infty)\times\overline{\Omega^c}),\,\varphi_t(0,\cdot)\in (\overline{\Omega^c})$.
\end{description}
\end{definition}

\begin{definition}[Global weak solution]\label{WS4} Let $u_0, u_1\in L^1_{\text{loc}}(\overline{\Omega^c})$ and $f\in L^1(\partial\Omega)$. We say that $u\in L^p_\text{loc}([0,\infty)\times\overline{\Omega^c})$  is a global weak solution to \eqref{P4}, if
\begin{equation}\label{WP4}\begin{split}
\int_Q|u|^{p}\varphi  dx dt+\int_{\Omega^c} (u_0(x)&+u_1(x))\varphi(0,x) dx-\int_{\Omega^c} u_0(x)\varphi_t(0,x) dx+\int_\Gamma f\varphi d\sigma dt
\\&= \int_Q u\varphi_{tt}dx dt-\int_Q u\varphi_{t}dx dt-\int_Q u\Delta\varphi dx dt,
\end{split}\end{equation}
holds for all $\varphi\in C_{t,x}^{2,2}((0,\infty)\times\Omega^c)\cap([0,\infty)\times\overline{\Omega^c}), \varphi\geq0$ such that
\begin{description}
\item[(i)] $\varphi(t,\cdot)\equiv0,\,\,t\geq T$ here $T>0$;
\item[(ii)] $\varphi(\cdot,x)\equiv0,\,\,|x|\geq R$ here $R>1$;
\item[(iii)] $\large\displaystyle\frac{\partial }{\partial \nu}\varphi(t,x)+\alpha \varphi(t,x)=0,\,\,(t,x)\in \Gamma$;
\item[(iv)] $\varphi_{t}, \varphi_{tt}, \Delta \varphi\in C([0,\infty)\times\overline{\Omega^c})$.
\end{description}
\end{definition}

\begin{theorem}\label{TT1} Suppose that $u_0, u_1\in L^1_{\text{loc}}(\overline{\Omega^c})$, $N\geq3$ and $f\in L^1(\partial \Omega)$. If
$$\int_{\partial\Omega}f(x)d\sigma>0\,\,\,\,\,\text{and}\,\,\,\,\,p=\frac{N}{N-2},$$then we have the following properties:
\begin{itemize}
\item[\textbf{(i)}] The problem \eqref{P1} does not admit positive weak solutions.
\item[\textbf{(ii)}] The problem \eqref{P2} possesses no global in time weak solutions.
\item[\textbf{(iii)}] The problem \eqref{P3} does not admit global in time weak solutions.
\item[\textbf{(iv)}] The problem \eqref{P4} does not admit global in time weak solutions.
\end{itemize}
\end{theorem}

\begin{remark} The following conclusions can be drawn from the above results:
\begin{itemize}
\item[(a)] Theorem \ref{TT1} completes the Theorems 1.2-1.4 from \cite{Ikeda}, where the critical cases $p=\frac{N}{N-2}$ of problems \eqref{P1}-\eqref{P4} has not been studied. Also, in case $\alpha\rightarrow\infty,$ the results of Theorem \ref{TT1} provide answers to open questions posed in \cite[Remark 1.7]{Jleli} and in \cite[Remark 1.6]{Zhang}.
\item[(b)] It is worth noting that in Theorem \ref{TT1} we did not assume that the initial data were positive. Therefore, when $\alpha\to 0,$ part (ii) of Theorem \ref{TT1} partially improves results of Zhang from \cite[part (a) of Theorem 1.3]{Zhang}, where similar results were obtained for positive initial data.
\end{itemize}
\end{remark}

\begin{theorem}\label{TT2} Suppose that $u_0, u_1\in L^1_{\text{loc}}(\overline{\Omega^c}), u_0\geq0$, $N\geq3,$ and $f\in L^1(\partial \Omega), f\geq 0$. We also assume that $u_0$ satisfies \eqref{u0}
and let $\lambda$ satisfies
$$\left\{\begin{array}{l}
2-N<\lambda<0,\,\,\,\,\,\,\,\,\,\,\,\,\,\,\,\,\,\,\text{if}\,\,\,\,\,N=3,4,\\{}\\
2-N<\lambda<4-N,\,\,\,\,\,\,\text{if}\,\,\,\,N\geq 5.
\end{array}\right.$$ If $$1<p=1-\frac{2}{\lambda},$$
then:
\begin{itemize}
\item[\textbf{(i)}] The problem \eqref{P2} does not admit global in time weak solutions.
\item[\textbf{(ii)}] The problem \eqref{P3} possesses no global in time weak solutions.
\item[\textbf{(iii)}] The problem \eqref{P4} does not admit global in time weak solutions.
\end{itemize}
\end{theorem}
\begin{remark} We present below some arguments related to the result of Theorem \ref{TT2}.
\begin{itemize}
\item[{(a)}]
Results of Theorem \ref{TT2} completes Theorem 1.6 from \cite{Ikeda}, where the critical case $p=1-\frac{2}{\lambda}$ was left open. In addition, we do not assume the positiveness of the initial function $u_1$, it can be a sign-changing function.
\item[(b)] It is easy to see that the function $f(x)$ does not affect the result of Theorem \ref{TT2}. Therefore, $1-\frac{2}{\lambda},\,\,\,\,2-N<\lambda<0,$ remains a critical exponent of the above problems with homogeneous Robin conditions.
\item[(c)] It is obvious that $$\frac{N}{N-2}<1-\frac{2}{\lambda},$$ for $2-N<\lambda<0.$  This means that if $u_0$ satisfies \eqref{u0}, then the critical exponent will be expanded, hence the set of solutions for which the studied problems are globally unsolvable also will expand.
\item[(d)] In Theorem \ref{TT2}, we have imposed the restriction from above on $\lambda$ for $N\geq 5$. The method that we use for the proof is not able to remove this restriction. Therefore, we do not know what will happen if $$1<p=1-\frac{2}{\lambda},\,\,4-N\leq\lambda<0,\,N\geq 5.$$ This question remains open.
\end{itemize}
\end{remark}
The proof of main results is based on methods of nonlinear capacity estimates specifically adapted to the nature of the exterior ball. Furthermore, the difference in our approach lies in the fact that we are considering a class of test functions with logarithmic arguments. This approach was previously successfully applied to the study of critical cases of the semilinear wave equation on the noncompact complete Riemannian manifold in \cite{Vetro} and the semilinear pseudo-parabolic equations on $\mathbb{R}^N$ in \cite{Torebek}.

\section{Test functions and useful estimates} This section will cover some properties of test functions. Additionally, we will establish some useful estimates related to the test functions.

We consider the function in the following form
\begin{equation}\label{TF0}
\varphi(t,x)=\varphi_1(t)\varphi_2(x),\,\,\,(t,x)\in Q,\end{equation} for sufficiently large $T,R$
\begin{equation}\label{TF1}
 \varphi_1(t)=\left(1-\frac{t}{T}\right)^l,\,\,\,t>0,\,\, l>\frac{2p}{p-1}
\end{equation}
and
\begin{equation}\label{TF2}
\varphi_2(x)=H(x)\xi(x)=H(x)\Psi^k\left(\frac{\ln\left(\frac{|x|}{\sqrt{R}}\right)}{\ln\left(\sqrt{R}\right)}\right),\,\,\,\,x\in\Omega^c,\,\,k>\frac{2p}{p-1},
\end{equation}
where $\Psi:\mathbb{R}\to [0,1]$ is a smooth function
\begin{equation}\label{TF4}
\Psi(s)=\left\{\begin{array}{l}
1,\,\,\,\,\,\,\text{if}\,\,\,\,-1\leq s\leq 0,\\
\searrow,\,\,\,\text{if}\,\,\,\,0< s<1,\\
0,\,\,\,\,\,\,\text{if}\,\,\,\,s\geq1.\end{array}\right.
\end{equation}
Furthermore, the harmonic function $H(x)$ defined in $\Omega^c$ by
\begin{equation}\label{TF5}
H(x)=\frac{N-2+\alpha}{\alpha}-|x|^{2-N},\,\,\,\text{if}\,\,\,\,N\geq3,
\end{equation}
which solves the exterior problem;
\begin{equation}\label{TF6}
\left\{\begin{array}{l}
-\Delta H=0\,\,\,\text{in}\,\,\,\,\Omega^c,\\{}\\
\large\displaystyle\frac{\partial}{\partial\nu}H+\alpha H=0\,\,\,\,\,\,\text{on}\,\,\,\,\partial\Omega.\end{array}\right.
\end{equation}

\begin{lemma}\label{LL1} For sufficiently large $T$ and $R$, the function $\varphi(t,x)$ is defined by \eqref{TF0} satisfies the following properties
\begin{description}
\item[(i)] $\varphi(t,\cdot)\equiv0,\,\,t\geq T$ here $T>0$;
\item[(ii)] $\varphi(\cdot,x)\equiv0,\,\,|x|\geq R$ here $R>1$;
\item[(iii)] $\large\displaystyle\frac{\partial }{\partial \nu}\varphi(t,x)+\alpha \varphi(t,x)=0,\,\,(t,x)\in \Gamma$, $\alpha>0$.
\end{description}
\end{lemma}
\begin{proof}[Proof of Lemma \ref{LL1}.]  $(i)-(ii)$ follows directly from the properties of $\varphi_1(t)$ and $\varphi_2(x)$, which given by \eqref{TF1}-\eqref{TF4}, respectively.

$(iii)$ In view of \eqref{TF0}, it implies that
\begin{equation*}\label{RC1}
\varphi_1(t)\biggl(\frac{\partial}{\partial \nu}\varphi_2(x)+\alpha\varphi_2(x)\biggr)=0.
\end{equation*}
Therefore, for $1<|x|<1+\varepsilon$, here $\varepsilon$ is sufficiently small, we obtain
\begin{equation*}\label{RC@}\begin{split}
 \nabla\varphi_2(x)&=\nabla\left(H(x)\xi(x)\right)
\\&=H(x)\nabla\xi(x)+\xi(x)\nabla H(x)
\\&=\nabla H(x). \end{split}\end{equation*}
Then, for sufficiently large $R$ by \eqref{TF2} for all $x\in\partial \Omega$, we get
\begin{equation*}\label{FL3}
\frac{\partial}{\partial\nu}\varphi_2(x)=\frac{\partial}{\partial\nu}H(x)=2-N.\end{equation*}
It is easy to verify by direct calculation that
$$\frac{\partial}{\partial \nu}\varphi_2(x)=-\alpha\varphi_2(x),$$
which completes the proof.
\end{proof}
%%%%%%%%%%%%%%%%%%%%%%%%%%%%%%%%%%
\begin{lemma}\label{X0}Let $N\geq3$ and $f\in L^1(\partial\Omega)$, then
\begin{equation*}\begin{split}
\int_\Gamma f\varphi d\sigma dt =C_{N,\alpha,l} T\int_{\partial\Omega} f \xi(x)d\sigma,\end{split}\end{equation*}
holds true, where
\begin{equation*}\begin{split}
C_{N,\alpha,l}=\frac{N-2}{\alpha(l+1)}.\end{split}\end{equation*}\end{lemma}

\begin{proof}[Proof of Lemma \ref{X0}.] Using properties of the test functions \eqref{TF0}-\eqref{TF5}, we get
\begin{equation*}\label{TF10}\begin{split}
\int_\Gamma f\varphi d\sigma dt =\int_\Gamma f\varphi_1(t)\varphi_2(x) d\sigma dt&=\left(\int_0^T\left(1-\frac{t}{T}\right)^ldt\right)\left(\int_{\partial \Omega}fH(x)\xi(x)d\sigma\right)
\\&=\left(\int_0^T\left(1-\frac{t}{T}\right)^ldt\right)\left(\int_{\partial \Omega}\left(\frac{N-2+\alpha}{\alpha}-1\right)f\xi(x)d\sigma\right)
\\&=\left(\frac{N-2}{\alpha(l+1)}\right)T\left(\int_{\partial \Omega}f\xi(x)d\sigma\right),\end{split}\end{equation*}
which gives the desired result.\end{proof}

\begin{lemma}\label{IE1} Let $N\geq3$ and $p=\frac{N}{N-2}$. Then for sufficiently large $T, R$ we have
\begin{equation}\label{IEF1}\begin{split}
&\mathcal{I}_1=\int_Q \varphi^{-\frac{1}{p-1}}|\varphi_t|^\frac{p}{p-1} dxdt \leq CT^{1-\frac{N}{2}}R^N,
\end{split}\end{equation}
\begin{equation}\label{QQ1}\begin{split}
&\mathcal{I}_2=\int_Q\varphi^{-\frac{1}{p-1}}|\varphi_{tt}|^\frac{p}{p-1}dx dt\leq CT^{1-N}R^N
\end{split}\end{equation}
and
\begin{equation}\label{IEF2}\begin{split}
\mathcal{I}_3&=\int_Q\varphi^{-\frac{1}{p-1}}|\Delta\varphi|^\frac{p}{p-1}dx dt
\leq CT\left((\ln R)^{1-N}+(\ln R)^{1-\frac{N}{2}}\right),
\end{split}\end{equation}
where $C$ are positive constants independent of $T$.
\end{lemma}
\begin{proof}[Proof of Lemma \ref{IE1}.] In view of \eqref{TF0} we obtain
\begin{equation*}\begin{split}
\mathcal{I}_1&=\left(\int_0^T \varphi_1^{-\frac{1}{p-1}}(t)\left|\varphi'_1(t)\right|^\frac{p}{p-1} dt\right) \left(\int_{\Omega^c} |\varphi_2(x)| dx \right).
\end{split}\end{equation*}
Consequently, from \eqref{TF1} with $l>\frac{2p}{p-1}$ it implies that
\begin{equation}\label{IEF3}\begin{split}
\int_0^T \varphi_1^{-\frac{1}{p-1}}(t)\left|\varphi'_1(t)\right|^\frac{p}{p-1} dt&=\int_0^T \left(1-\frac{t}{T}\right)^{-\frac{l}{p-1}}\left|lT^{-1}\left(1-\frac{t}{T}\right)^{l-1}\right|^\frac{p}{p-1} dt
\\&=\frac{(p-1)l^{\frac{p}{p-1}}}{l(p-1)-1}T^{1-\frac{p}{p-1}}
\\&=CT^{1-\frac{N}{2}},\end{split}\end{equation}since $\frac{p}{p-1}=\frac{N}{2}.$
\\Therefore, by \eqref{TF2} and \eqref{TF4} for a sufficiently large $R>1$ we have
\begin{equation}\label{IEF4}\begin{split}
 \int_{\Omega^c} |\varphi_2(x)| dx &=  \int_{\Omega^c} \left(\large\displaystyle\frac{N-2+\alpha}{\alpha}-|x|^{2-N}\right)\Psi^k\left(\frac{\ln\left(\frac{|x|}{\sqrt{R}}\right)}{\ln\left(\sqrt{R}\right)}\right)dx
\\&\stackrel{|x|=r}\leq\int_{1<r<R} \left(\frac{N-2+\alpha}{\alpha}-r^{2-N}\right)r^{N-1}dr
\\&\leq\left(\frac{N-2+\alpha}{\alpha}\right)\int_{1<r<R}r^{N-1}dr
\\&=\left(\frac{N-2+\alpha}{\alpha N}\right)R^{N}-\left(\frac{N-2+\alpha}{\alpha N}\right)
\\&\leq C R^N. \end{split}\end{equation}
From the combination of \eqref{IEF3} and \eqref{IEF4} we get \eqref{IEF1}.
\\Similarly, we arrive at
\begin{equation*}\begin{split}
&\mathcal{I}_2=\left(\int_0^T \varphi_1^{-\frac{1}{p-1}}(t)|\varphi_1''(t)|^\frac{p}{p-1} dt\right) \left(\int_{\Omega^c} |\varphi_2(x)| dx \right).
\end{split}\end{equation*}
Using \eqref{TF1} with $l>\frac{2p}{p-1}$ we get
\begin{equation}\label{Q4}\begin{split}
\int_0^T \varphi_1^{-\frac{1}{p-1}}(t)\left|\varphi''_1(t)\right|^\frac{p}{p-1} dt&=\int_0^T \left(1-\frac{t}{T}\right)^{-\frac{l}{p-1}}\left|l(l-1)T^{-2}\left(1-\frac{t}{T}\right)^{l-2}\right|^\frac{p}{p-1} dt
\\&=\frac{(p-1)(l(l-1))^\frac{p}{p-1}}{l(p-1)-p-1}T^{1-\frac{2p}{p-1}}
\\&=CT^{1-N},
\end{split}\end{equation}here we have used $\frac{2p}{p-1}=N.$
Combining \eqref{IEF4} and \eqref{Q4}  it follows \eqref{QQ1}.
\\Next, let us calculate the next integral
\begin{equation*}\label{POP5}\begin{split}
&\mathcal{I}_3=\left(\int_0^T |\varphi_1(t)| dt\right) \left(\int_{\Omega^c} \varphi_2^{-\frac{1}{p-1}}(x)|\Delta\varphi_2(x)|^\frac{p}{p-1} dx \right).
\end{split}\end{equation*}
From \eqref{TF1}, it is  easy to verify by a direct calculation that
\begin{equation}\label{IEF5}\begin{split}
\int_0^T|\varphi_1| dt&=\int_0^T\left(1-\frac{t}{T}\right)^ldt
=\frac{1}{l+1}T.\end{split}\end{equation}
Since, the function $\xi(x)$ is a radial, we obtain

\begin{equation}\label{Q1}\begin{split}
|\Delta\xi(r)|&\leq C\left[\frac{1}{r^2\ln^2\sqrt{R}}\Psi^{k-2}\left(\frac{\ln\left(\frac{r}{\sqrt{R}}\right)}{\ln\left(\sqrt{R}\right)}\right)+\frac{1}{r^2\ln\sqrt{R}}\Psi^{k-1}\left(\frac{\ln\left(\frac{r}{\sqrt{R}}\right)}{\ln\left(\sqrt{R}\right)}\right)\right],
\end{split}\end{equation} where $r=|x|=(x_1^2+x_2^2+...+x_n^2)^\frac{1}{2}$ and $C>0$ is an arbitrary constant.
\\Therefore, taking into account the function $H(x)$ is harmonic and remaining \eqref{Q1}, we arrive at
\begin{equation}\label{Q2}\begin{split}
|\Delta\varphi_2(x)|&=|\Delta\left[H(x)\xi(x)\right]|\\&\leq H(x)|\Delta\left[\xi(x)\right]|+2|\nabla H(x)||\nabla\xi(x)|
\\&\leq CH(x)\left[\frac{1}{|x|^2\ln^2\sqrt{R}}\Psi^{k-2}\left(\frac{\ln\left(\frac{|x|}{\sqrt{R}}\right)}{\ln\left(\sqrt{R}\right)}\right)+\frac{1}{|x|^2\ln\sqrt{R}}\Psi^{k-1}\left(\frac{\ln\left(\frac{|x|}{\sqrt{R}}\right)}{\ln\left(\sqrt{R}\right)}\right)\right]\\&+\frac{2k(N-2)}{|x|^2\ln\sqrt{R}}|x|^{2-N}\Psi^{k-1}\left(\frac{\ln\left(\frac{|x|}{\sqrt{R}}\right)}{\ln\left(\sqrt{R}\right)}\right)
\\&\leq  CH(x)\left(\frac{1}{|x|^2\ln^2\sqrt{R}}\xi^{\frac{k-2}{k}}\left(x\right)+\frac{1}{|x|^2\ln\sqrt{R}}\xi^{\frac{k-1}{k}}\left(x\right)\right),
\end{split}\end{equation}
thanks to $|x|^{2-N}\leq \alpha H(x).$
\\Consequently, we get
\begin{equation*}\label{Q3}\begin{split}
&\int_{\Omega^c} \varphi_2^{-\frac{1}{p-1}}(x)|\Delta\varphi_2(x)|^\frac{p}{p-1} dx
\\&\leq C\int_{\Omega^c} \left|H(x)\xi(x)\right|^{-\frac{1}{p-1}}\left|H(x)\left[\frac{1}{|x|^2\ln^2\sqrt{R}}\xi^{\frac{k-2}{k}}\left(x\right)+\frac{1}{|x|^2\ln\sqrt{R}}\xi^{\frac{k-1}{k}}\left(x\right)\right]
\right|^\frac{p}{p-1}   dx.\end{split}\end{equation*}
Using Schwarz's inequality and the following inequality
\begin{equation*}
(a+b)^m\leq 2^{m-1}(a^m+b^m),\,\,\,a\geq0,\,b\geq0,\,m=\frac{p}{p-1},\end{equation*}we obtain
\begin{equation*}\begin{split}
\int_{\Omega^c} \varphi_2^{-\frac{1}{p-1}}(x)|\Delta\varphi_2(x)|^\frac{p}{p-1} dx
&\leq C\int_{\Omega^c} H(x) |\xi(x)|^{-\frac{1}{p-1}}\left[\frac{1}{|x|^2\ln^2\sqrt{R}}\left|\xi^{\frac{k-2}{k}}\right|\left(x\right)\right]^\frac{p}{p-1}   dx\\& +C\int_{\Omega^c} H(x)|\xi(x)|^{-\frac{1}{p-1}}\left[\frac{1}{|x|^2\ln\sqrt{R}}\left|\xi^{\frac{k-1}{k}}\right|\left(x\right)\right]^\frac{p}{p-1}   dx.\end{split}\end{equation*}
Therefore, noting that $k>\frac{2p}{p-1}$ and \eqref{TF4} we arrive at
 \begin{equation*}\begin{split}
\int_{\Omega^c} \varphi_2^{-\frac{1}{p-1}}(x)|\Delta\varphi_2(x)|^\frac{p}{p-1} dx
&\leq C\int_{\sqrt{R}<|x|<R} H(x)\left[\frac{1}{|x|^2\ln^2\sqrt{R}}\right]^\frac{p}{p-1}dx
\\&+C\int_{\sqrt{R}<|x|<R}H(x)\left[\frac{1}{|x|^2\ln\sqrt{R}}\right]^\frac{p}{p-1}dx.\end{split}\end{equation*}
At this stage, recalling $p=\frac{N}{N-2}$, we calculate the right-hand side of the last estimate separately
 \begin{equation}\label{PPQ}\begin{split}
\int_{\sqrt{R}<|x|<R} H(x)\left[\frac{1}{|x|^2\ln^2\sqrt{R}}\right]^\frac{p}{p-1}dx&\stackrel{|x|=r}{=}\left[\frac{1}{\ln^2\sqrt{R}}\right]^\frac{N}{2}\int_{\sqrt{R}<r<R}\left(\large\displaystyle\frac{N-2+\alpha}{\alpha}-r^{2-N}\right)r^{-1}dr
\\&=\left[\frac{1}{\ln^2\sqrt{R}}\right]^\frac{N}{2} \int_{\sqrt{R}<r<R} \left(\large\displaystyle\frac{N-2+\alpha}{\alpha}r^{-1}-r^{1-N}\right)dr
\\&\leq C(\ln R)^{1-N}\end{split}\end{equation}
and
\begin{equation}\label{PPQ1}\begin{split}
\int_{\sqrt{R}<|x|<R}H(x)\left[\frac{k(N-2)}{|x|^2\ln\sqrt{R}}\right]^\frac{p}{p-1}dx&\leq C(\ln R)^{-\frac{N}{2}+1}.\end{split}\end{equation}
Consequently, due to the last inequalities, we obtain
 \begin{equation}\label{Q14}\begin{split}
\int_{\Omega^c} \varphi_2^{-\frac{1}{p-1}}(x)&|\Delta\varphi_2(x)|^\frac{p}{p-1} dx\leq C\left((\ln R)^{1-N}+(\ln R)^{1-\frac{N}{2}}\right).\end{split}\end{equation}
Therefore, combining \eqref{IEF5} and \eqref{Q14} we get \eqref{IEF2}, which completes the proof.\end{proof}
Next, we will introduce the test function in the following form
\begin{equation}\label{QWE1}
 \psi=\psi_1(t)\psi_2(x)=\left(1-\frac{t}{T}\right)^lH(x)\Psi\left(\frac{|x|}{R}\right),\,\, l>\frac{2p}{p-1},
\end{equation}
where $T,R>0$ and $H(x)$ is defined by \eqref{TF5}.

The function ${\Psi\in C^2(\mathbb{R}_+)}$ is the standard cut-off function given by
\begin{equation}\label{TTQ2}
\Psi\left(\frac{|x|}{R}\right)=
 \begin{cases}
   1 &\text{if\,\,\,\, $\,\,\,\, 0\leq |x|\leq R$},\\
   \searrow &\text{if\,\,\,\, $R<|x|<2R$},\\
   0 &\text{if\,\,\,\, $2R\leq|x|$}.
 \end{cases}\end{equation}

\begin{corollary}\label{C1} Let $N\geq3$ and $p=\frac{\lambda-2}{\lambda}$. Then for sufficiently large $T, R$ there hold true
\begin{equation}\label{IEFQ1}\begin{split}
&\mathcal{J}_1=\int_Q \psi^{-\frac{1}{p-1}}|\psi_t|^\frac{p}{p-1} dxdt \leq CT^{\frac{\lambda}{2}}R^{N},
\end{split}\end{equation}
\begin{equation}\label{QQQ1}\begin{split}
&\mathcal{J}_2=\int_Q\psi^{-\frac{1}{p-1}}|\psi_{tt}|^\frac{p}{p-1}dx dt\leq CT^{\lambda-1}R^{N}
\end{split}\end{equation}
and
\begin{equation}\label{IEFQ2}\begin{split}
\mathcal{J}_3&=\int_Q\psi^{-\frac{1}{p-1}}|\Delta\psi|^\frac{p}{p-1}dx dt\leq CTR^{N+\lambda-2},
\end{split}\end{equation}
where $C$ are positive constants independent of $T$ and $R$.
\end{corollary}
\begin{proof}[Proof of Corollary \ref{C1}.]
According to the test function $\psi=\psi_1(t)\psi_2(x)$, we have
\begin{equation*}\begin{split}
\mathcal{J}_1&=\left(\int_0^T \psi_1^{-\frac{1}{p-1}}(t)\left|\psi'_1(t)\right|^\frac{p}{p-1} dt\right) \left(\int_{\Omega^c} |\psi_2(x)| dx \right)
\end{split}\end{equation*}and
\begin{equation*}\begin{split}
\mathcal{J}_2&=\left(\int_0^T \psi_1^{-\frac{1}{p-1}}(t)\left|\psi''_1(t)\right|^\frac{p}{p-1} dt\right) \left(\int_{\Omega^c} |\psi_2(x)| dx \right).
\end{split}\end{equation*}
Using \eqref{IEF3} and \eqref{Q4} to the first part of the last integrals with $$\frac{p}{p-1}=-\frac{\lambda-2}{2}$$ we obtain
\begin{equation}\label{IEF30}\begin{split}
\int_0^T \psi_1^{-\frac{1}{p-1}}(t)\left|\psi'_1(t)\right|^\frac{p}{p-1} dt&=CT^{\frac{\lambda}{2}},\end{split}\end{equation}
\begin{equation}\label{Q40}\begin{split}
\int_0^T \psi_1^{-\frac{1}{p-1}}(t)\left|\psi''_1(t)\right|^\frac{p}{p-1} dt=CT^{\lambda-1}.
\end{split}\end{equation}
In addition, we have
\begin{equation*}\begin{split}
 \int_{\Omega^c} |\psi_2(x)| dx &=  \int_{\Omega^c} \left(\large\displaystyle\frac{N-2+\alpha}{\alpha}-|x|^{2-N}\right)\Psi\left(\frac{|x|}{R}\right)dx
\\&\stackrel{|x|=r}\leq\int_{1< r<2} \left(\frac{N-2+\alpha}{\alpha}-r^{2-N}\right)r^{N-1}dr
\\&\leq\left(\frac{N-2+\alpha}{\alpha}\right)\int_{1< r<2}r^{N-1}dr
\\&\leq C R^{N}. \end{split}\end{equation*}
From the combination of \eqref{IEF30} and \eqref{Q40} with \eqref{Q4} we get the estimates for $\mathcal{J}_1, \mathcal{J}_2$, respectively.
\\Next, we consider the integral $\mathcal{J}_3$. In view of \eqref{PPQ} and  \eqref{PPQ1} with $$\frac{p}{p-1}=-\frac{\lambda-2}{2}$$  it follows that
 \begin{equation*}\begin{split}
\int_{\Omega^c}\psi_2^{-\frac{1}{p-1}}|\Delta\varphi_2|^\frac{p}{p-1}dx&=\int_{R<|x|<2R}\Psi^{-\frac{1}{p-1}}\left(\frac{|x|}{R}\right)\left|\Delta\Psi\left(\frac{|x|}{R}\right)\right|^\frac{p}{p-1}dx
\\&=R^{-\frac{2 p}{p-1}}\int_{R<|x|<2R}\Psi^{-\frac{1}{p-1}}\left(\frac{|x|}{R}\right)\left|\Psi''\left(\frac{|x|}{R}\right)\right|^\frac{p}{p-1}dx
\\&\stackrel{|x|=yR}=R^{-\frac{2 p}{p-1}+N}\int_{1<|y|<2}\Psi^{-\frac{1}{p-1}}\left(|y|\right)\left|\Psi''\left(|y|\right)\right|^\frac{p}{p-1}dy
\\&\leq C R^{-\frac{ 2p}{p-1}+N}
\\&= C R^{N+\lambda-2}.\end{split}\end{equation*}
Combining the expression \eqref{IEF5} with the last estimates, we get \eqref{IEFQ2}.\end{proof}
%%%%%%%%%%%%%%%%%%%%

\section{Proof of basic theorems} In this subsection we will prove the main theorems in detail.

Note that in a similar way one can prove that the results of Theorem \ref{TT1} remain true if $\alpha\rightarrow+\infty$ and $\alpha=0$. To do this, it suffices to consider instead of Robin harmonic function $H(x)$, the Dirichlet harmonic function $1-|x|^{2-n}$ and some constant $C=const>0$, respectively.
\begin{proof}[Proof of Theorem \ref{TT1}.]\textbf{(i)} We argue by contradiction by supposing that a global weak solution $u\in L^p_\text{loc}(\overline{\Omega^c})$ to problem \eqref{P1}. From Definition \ref{WS1} it yields
\begin{equation*}\begin{split}
&\int_{\Omega^c}u^{p}\varphi_2  dx +\int_\Gamma f\varphi_2 d\sigma \leq \int_{\Omega^c} u|\Delta\varphi_2| dx.
\end{split}\end{equation*}
Hence, using H\"{o}lder's and  the $\varepsilon$-Young inequality with $\varepsilon=p$, we arrive at
\begin{equation*}\begin{split}
\int_{\Omega^c} u|\Delta\varphi_2| dx&=\int_{\Omega^c} u|\varphi_2|^{\frac{1}{p}}|\Delta\varphi_2| |\varphi_2|^{-\frac{1}{p}} dx
\\&\leq \biggl(\int_{\Omega^c} u^p|\varphi_2| dx\biggr)^\frac{1}{p}\biggl(\int_{\Omega^c} \varphi^{-\frac{1}{p-1}}|\Delta\varphi_2|^{\frac{p}{p-1}} dx\,\,\biggr)^\frac{p-1}{p}.
\\&\leq \int_{\Omega^c}u^{p}\varphi_2  dx
+(p-1)p^{-\frac{p}{p-1}}\int_{\Omega^c}\varphi_2^{-\frac{1}{p-1}}\left|\Delta\varphi_2\right|^{\frac{p}{p-1}}dx .
\end{split}\end{equation*}
From the last inequality, we obtain
\begin{equation*}\begin{split}
\int_\Gamma f\varphi_2 d\sigma \leq C\int_{\Omega^c}\varphi_2^{-\frac{1}{p-1}}\left|\Delta\varphi_2\right|^{\frac{p}{p-1}}dx.
\end{split}\end{equation*}
In view of  \eqref{Q14} with $R\to \infty$, we get the contradiction $$\int_{\partial\Omega}f(x)d\sigma>0.$$

\textbf{(ii)} The proof is done by contradiction. Assume that there exists a global in time weak solution $u\in L^p_\text{loc}([0,\infty)\times\overline{\Omega^c})$  of the problem \eqref{P2}. In view of Definition \ref{WS2}, we have
\begin{equation}\label{MM1}\begin{split}
&\int_Q|u|^{p}\varphi  dx dt +\int_{\Omega^c} u_0(x)\varphi(0,x) dx+\int_\Gamma f\varphi d\sigma dt
\leq \int_Q |u||\varphi_t|dx dt+\int_Q |u||\Delta\varphi| dx dt.
\end{split}\end{equation}
\\Therefore, by H\"{o}lder's inequality, we obtain
\begin{equation*}\begin{split}
 \int_Q |u||\varphi_t|dx dt&=\int_Q |u|\varphi^{\frac{1}{p}}|\varphi_t|\varphi^{-\frac{1}{p}}dx dt\\&\leq\biggl(\int_Q |u|^p\varphi dx dt\biggr)^\frac{1}{p}\biggl(\,\,\underbrace{\int_Q \varphi^{-\frac{1}{p-1}}|\varphi_t|^\frac{p}{p-1}dx dt}_{\mathcal{I}_1}\,\,\biggr)^\frac{p-1}{p}
\end{split}\end{equation*}and
\begin{equation*}\begin{split}
\int_Q|u||\Delta\varphi|dx dt&=\int_Q |u|\varphi^{\frac{1}{p}}|\varphi_t|\varphi^{-\frac{1}{p}}dx dt\\&\leq \biggl(\int_Q |u|^p\varphi dx dt\biggr)^\frac{1}{p}\biggl(\,\,\underbrace{\int_Q\varphi^{-\frac{1}{p-1}}|\Delta\varphi|^\frac{p}{p-1}dx dt}_{\mathcal{I}_3}\,\,\biggr)^\frac{p-1}{p}.
\end{split}\end{equation*}
Using the $\varepsilon$-Young inequality with $\displaystyle\varepsilon=\frac{p}{2}$ in the last inequalities, it follows that
\begin{equation*}\begin{split}
\int_Q|u|\left|\varphi_t\right|dx dt&\leq\frac{1}{2}\int_Q|u|^{p}\varphi  dx dt
+\frac{p-1}{p}\biggl(\frac{p}{2}\biggr)^{-\frac{1}{p-1}}\underbrace{\int_Q\varphi^{-\frac{1}{p-1}}\left|\varphi_t\right|^{\frac{p}{p-1}}dx dt}_{\mathcal{I}_1}. \end{split}\end{equation*}
Similarly, one obtains
\begin{equation*}\begin{split}
\int_Q|u|\left|\Delta\varphi_t\right|dx dt&\leq\frac{1}{2}\int_Q|u|^{p}\varphi  dx dt
+\frac{p-1}{p}\biggl(\frac{p}{2}\biggr)^{-\frac{1}{p-1}}\underbrace{\int_Q\varphi^{-\frac{1}{p-1}}\left|\Delta\varphi\right|^{\frac{p}{p-1}}dx dt}_{\mathcal{I}_3}.
\end{split}\end{equation*}
Consequently, we can rewrite \eqref{MM1} as
\begin{equation}\label{PL5}\begin{split}
\int_{\Omega^c}u_0(x)\varphi(0,x)dx+\int_\Gamma f\varphi d\sigma dt\leq C(p)(\mathcal{I}_1+\mathcal{I}_3),
\end{split}\end{equation}where $\displaystyle C(p)=\frac{p-1}{p}\left(\frac{p}{2}\right)^{-\frac{1}{p-1}}.$

Then, from Lemma \ref{X0} and Lemma \ref{IE1}, for $T=R^j, j>0,$ we get
\begin{equation*}\label{PL6}\begin{split}
\int_{\partial \Omega}f\xi(x) d\sigma&\leq CR^{-j}\int_{\Omega^c}|u_0(x)|\varphi(0,x)dx
\\&+ C(p)\left(R^{N(-\frac{j}{2}+1)}+(\ln R)^{-N+1}+(\ln R)^{-\frac{N}{2}+1}\right),
\end{split}\end{equation*} where $C>0$ is an arbitrary constant.

Since
\begin{align*}\lim\limits_{R\to\infty}H(x)\xi(x)&=\frac{N-2}{\alpha}\lim\limits_{R\to\infty}\Psi^k\left(\frac{\ln\left(\frac{|x|}{\sqrt{R}}\right)}{\ln\left(\sqrt{R}\right)}\right)\\& =\frac{N-2}{\alpha}\Psi^k\left(-1\right)\\&=\frac{N-2}{\alpha},\,x\in\partial\Omega,
\end{align*}
taking $j>2$ and passing to the limit as $R\to\infty$ in the above inequality one can obtains
\begin{align}\label{QQQ}&\lim\limits_{R\to\infty}\int_{\partial \Omega}f\xi(x) d\sigma=\int_{\partial \Omega}fd\sigma\leq 0,
\end{align}
which is a contradiction with $$\int_{\partial \Omega}fd\sigma>0.$$

\textbf{(iii)} At this stage, following the same technique as used in the proof of the previous case, we deduce that
\begin{equation}\label{PW1}\begin{split}
\int_{\Omega^c} u_1(x)\varphi(0,x) dx&-\int_{\Omega^c} u_0(x)\varphi_t(0,x) dx+\int_\Gamma f\varphi d\sigma dt\leq  C(p)(\mathcal{I}_2+\mathcal{I}_3).
\end{split}\end{equation}
Using Lemma \ref{X0} and Lemma \ref{IE1}, for $T=R^j, j>0,$ we obtain
\begin{equation*}\begin{split}
\int_{\partial \Omega}f\xi(x) d\sigma&\leq CR^{-j}\int_{\Omega^c} |u_1(x)|\varphi(0,x) dx+CR^{-j}\int_{\Omega^c} |u_0(x)|\varphi_t(0,x) dx\\&+C(p)\left(R^{N(-j+1)}+(\ln R)^{-N+1}
+(\ln R)^{-\frac{N}{2}+1}\right).
\end{split}\end{equation*}
Finally, taking $j>1$ and passing to the limit $R\to\infty$ and noting \eqref{QQQ}, we get a contradiction with $$\int_{\partial \Omega}fd\sigma>0.$$

\textbf{(iv)} Next, acting in the same way as in the above case, we get the following estimate
\begin{equation*}\begin{split}
\int_{\Omega^c} (u_0(x)+u_1(x))\varphi(0,x) dx&-\int_{\Omega^c} u_0(x)\varphi_t(0,x) dx+\int_\Gamma f\varphi d\sigma dt\\&\leq  C(p)(\mathcal{I}_1+\mathcal{I}_2+\mathcal{I}_3),
\end{split}\end{equation*}where $C(p)=\frac{p-1}{p}\biggl(\frac{p}{3}\biggr)^{-\frac{1}{p-1}}.$
\\From the results of Lemma \ref{X0} and Lemma \ref{IE1} with changing $T=R^j, j>0,$ we arrive at
\begin{equation*}\begin{split}
\int_{\partial \Omega}f\xi(x) d\sigma&\leq CR^{-j}\int_{\Omega^c} (|u_0(x)|+|u_1(x)|)\varphi(0,x) dx+CR^{-j}\int_{\Omega^c} |u_0(x)|\varphi_t(0,x) dx\\&+C(p)\left(R^{N(-\frac{j}{2}+1)}+R^{N(-j+1)}+(\ln R)^{-N+1}
+(\ln R)^{-\frac{N}{2}+1}\right).
\end{split}\end{equation*}
Hence, taking $j>2$ and passing to the limit $R\to\infty$ in the last inequality we get a contradiction with $$\int_{\partial \Omega}fd\sigma>0,$$ which completes the proof.\end{proof}

\begin{proof}[Proof of Theorem \ref{TT2}.]\textbf{(i)} The proof also will be done by contradiction. Suppose that $u\in L^p_\text{loc}([0,\infty)\times\overline{\Omega^c})$  is a global weak solution to \eqref{P2}.

Therefore, from Definition \ref{WS2} and using H\"{o}lder's, $\varepsilon$-Young's inequality with fact that $f\geq0$, we obtain instead of the estimate \eqref{PL5}
\begin{equation*}\label{QT}\begin{split}
\int_{\Omega^c}u_0(x)\psi(0,x)dx\leq C(p)\left(\int_Q\psi^{-\frac{1}{p-1}}\left|\psi_t\right|^{\frac{p}{p-1}}dx dt+\int_Q\psi^{-\frac{1}{p-1}}\left|\Delta\psi\right|^{\frac{p}{p-1}}dx dt\right),
\end{split}\end{equation*}where $\displaystyle C(p)=\frac{p-1}{p}\left(\frac{p}{2}\right)^{-\frac{1}{p-1}}$.
\\Hence, in view of $p=\frac{\lambda-2}{\lambda}$ from  Corollary \ref{C1}, one obtains
\begin{equation}\label{QT6}\begin{split}
\int_{\Omega^c}u_0(x)\psi(0,x)dx&\leq C(\lambda)(\mathcal{J}_1+\mathcal{J}_3)
\\&\leq C(\lambda)\left(T^{\frac{\lambda}{2}}R^{N}+TR^{N+\lambda-2}\right),
\end{split}\end{equation}here $\displaystyle C(\lambda)=\frac{2}{2-\lambda}\left(\frac{\lambda-2}{2\lambda}\right)^{\frac{\lambda}{2}}$ is a positive constant.

At this stage,  recalling \eqref{QWE1} with \eqref{TTQ2}, we get
\begin{equation}\label{QT2}\begin{split}
\int_{\Omega^c}u_0(x)\psi(0,x)dx&=\int_{\Omega^c}u_0(x)\left(\frac{N-2+\alpha}{\alpha}-|x|^{2-N}\right)\Psi\left(\frac{|x|}{R}\right)dx
\\&\geq \frac{A}{\alpha}\int\limits^{R}_{{R}^\frac{1}{2}}|x|^{\lambda}dx
\\&\geq \frac{A}{\alpha} R^{\frac{\lambda+N}{2}}>0.
\end{split}\end{equation}
Next, using the last estimate and changing the variable $T=R^j$, we can rewrite \eqref{QT6} in the following form
\begin{equation}\label{QT3}\begin{split}
0<A &\leq \alpha C(\lambda)R^{-\frac{\lambda+N}{2}}\left(R^\frac{j\lambda}{2}R^{N}+R^jR^{N+\lambda-2}\right)
\\&=\alpha C(\lambda)\left(R^{\frac{j\lambda-\lambda+N}{2}}+R^{\frac{2j+\lambda+N-4}{2}}\right),
\end{split}\end{equation} where $C(\lambda)>0$ is a positive constant independent of $R.$

Let us assume that
$$\left\{\begin{array}{l}
\lambda<0,\,\,\,\,\,\,\,\,\,\,\,\,\,\,\text{if}\,\,\,N=3,4,\\{}\\
\lambda<4-N,\,\,\text{if}\,\,N\geq 5.
\end{array}\right.$$
Then, choosing $$0<j<\min\biggl\{1-\frac{N}{\lambda}, \frac{4-N-\lambda}{2}\biggr\}$$ we make sure that $$\frac{j\lambda-\lambda+N}{2}<0\,\,\,\text{and}\,\,\,\frac{2j+\lambda+N-4}{2}<0.$$
Note that since $$1-\frac{N}{\lambda}>0\,\,\,\,\text{and} \,\,\,\,\frac{4-N-\lambda}{2}>0,$$ the set $\left(0,\min\biggl\{1-\frac{N}{\lambda}, \frac{4-N-\lambda}{2}\biggr\}\right)$ is not empty.

Hence, passing to the limit $R\to\infty$ in \eqref{QT3} we get a contradiction with $A>0$.
\\\textbf{(ii)} As previously, we argue by contradiction.
Assume that $u\in L^p_\text{loc}([0,\infty)\times\overline{\Omega^c})$  is a global weak solution to \eqref{P3}. Following a similar argument as that used in the proof of part $\text{(i)}$, we obtain
\begin{equation*}\begin{split}
\int_{\Omega^c} u_1(x)\psi(0,x) dx-\int_{\Omega^c} u_0(x)\psi_t(0,x) dx&\leq C(\lambda)(\mathcal{J}_2+\mathcal{J}_3)
\\&\leq C(\lambda)\left(T^{{\lambda-1}}R^{N}+TR^{N+\lambda-2}\right).
\end{split}\end{equation*}
Next, from the elementary calculation
\begin{equation}\label{u01}
-\int_{\Omega^c} u_0(x)\psi_t(0,x) dx=CT^{-1}\int_{\Omega^c} u_0(x)\psi(0,x) dx,
\end{equation} we get
\begin{equation*}\begin{split}
\int_{\Omega^c} u_0(x)\psi(0,x) dx&\leq C(\lambda)\left(T^{{\lambda}}R^{N}+T^2R^{N+\lambda-2}\right)-T\int_{\Omega^c} u_1(x)\psi(0,x) dx.
\end{split}\end{equation*}
Using \eqref{QT2}
\begin{equation*}\begin{split}
0<\frac{A}{\alpha} R^{\frac{\lambda+N}{2}}&\leq C(\lambda)\left(T^{{\lambda}}R^{N}+T^2R^{N+\lambda-2}\right)-T\int_{\Omega^c} u_1(x)\psi_2(x) dx\end{split}\end{equation*}
and changing $T=R^j,\,\,j>0$, we deduce that
\begin{equation*}\begin{split}
0<A &\leq \alpha C(\lambda)R^{-\frac{\lambda+N}{2}}\left(R^{j\lambda+N}+R^{2j+N+\lambda-2}\right)-\alpha R^j\int_{\Omega^c} u_1(x)\psi_2(x) dx,
\end{split}\end{equation*}
which is
\begin{equation}\label{+}\begin{split}
0<A &\leq \alpha C(\lambda)\left(R^{\frac{2j\lambda- \lambda+N}{2}}+R^{\frac{4j+ \lambda+N-4}{2}}\right)-\alpha R^\frac{2j-\lambda-N}{2}\int_{\Omega^c} u_1(x)\psi_2(x) dx,
\end{split}\end{equation} where $C(\lambda)>0$ is a positive constant independent of $R.$

Let us assume that
$$\left\{\begin{array}{l}
\lambda<0,\,\,\,\,\,\,\,\,\,\,\,\,\,\,\text{if}\,\,\,N=3,4,\\{}\\
\lambda<4-N,\,\,\text{if}\,\,N\geq 5.
\end{array}\right.$$
Choosing $$0<j<\min\biggl\{\frac{\lambda-N}{2\lambda}, \frac{4-N-\lambda}{4}, \frac{\lambda+N}{2}\biggr\}$$ we make sure that $$\frac{2j-\lambda-N}{2}<0,\,\,\frac{2j\lambda- \lambda+N}{2} <0\,\,\,\text{and}\,\,\,\frac{4j+ \lambda+N-4}{2}<0.$$
Due to the fact that $$\frac{\lambda+N}{2}>0,\,\,\frac{\lambda-N}{2\lambda}>0\,\,\,\,\text{and} \,\,\,\,\frac{4-N-\lambda}{4}>0,$$ the set $\left(0,\min\biggl\{\frac{\lambda-N}{2\lambda}, \frac{4-N-\lambda}{4}, \frac{\lambda+N}{2}\biggr\}\right)$ is not empty.

Finally, passing to the limit $R\to\infty$ in \eqref{+} we get a contradiction with $A>0$.
\\\textbf{(iii)} Suppose that $u\in L^p_\text{loc}([0,\infty)\times\overline{\Omega^c})$  is a global weak solution to \eqref{P2}. Continuing exactly as in the proof part $\textbf{(i)}$, we arrive at
\begin{equation*}\label{QT11}\begin{split}
\int_{\Omega^c} (u_0(x)+u_1(x))\psi(0,x) dx&-\int_{\Omega^c} u_0(x)\psi_t(0,x) dx\leq C(\lambda)(\mathcal{J}_1+\mathcal{J}_2+\mathcal{J}_3).
\end{split}\end{equation*}
From \eqref{u01}, it follows that
\begin{equation*}\label{QTR11}\begin{split}
\int_{\Omega^c} (u_0(x)+u_1(x))\psi(0,x) dx&+CT^{-1}\int_{\Omega^c} u_0(x)\psi(0,x) dx\leq C(\lambda)(\mathcal{J}_1+\mathcal{J}_2+\mathcal{J}_3).
\end{split}\end{equation*}
Therefore, in view of \eqref{QT2} and Corollary \ref{C1}, we have
\begin{equation*}\begin{split}
\left(1+\frac{1}{T}\right)\frac{A}{\alpha} R^{\frac{\lambda+N}{2}}&\leq C(\lambda)\left(T^{\frac{\lambda}{2}}R^{N}+T^{{\lambda-1}}R^{N}+TR^{N+\lambda-2}\right)\\&-\int_{\Omega^c} u_1(x)\psi_2(x) dx.\end{split}\end{equation*}
At this stage, using  and changing the variable $T=R^j$ in the last inequality, we obtain
\begin{equation*}\begin{split}
\left(1+\frac{1}{R^j}\right)A &\leq \alpha C(\lambda)R^{-\frac{\lambda+N}{2}}\left(R^{\frac{\lambda}{2}j+N}+R^{(\lambda-1)j+N}+R^{j+N+\lambda-2}\right)\\&-\alpha R^{-\frac{\lambda+N}{2}}\int_{\Omega^c} u_1(x)\psi_2(x) dx.\end{split}\end{equation*}
After doing some elementary calculations, we arrive at
\begin{equation}\label{++}\begin{split}
0<A &\leq \alpha C(\lambda)\left(R^{\frac{\lambda j- \lambda+N}{2}}+R^{\frac{2(\lambda-1) j- \lambda+N}{2}}+R^{\frac{2j+ \lambda+N-4}{2}}\right)\\&-\alpha R^{-\frac{\lambda+N}{2}}\int_{\Omega^c} u_1(x)\psi_2(x) dx,\end{split}\end{equation}
where $C(\lambda)>0$ is a positive constant independent of $R.$

Suppose that
$$\left\{\begin{array}{l}
\lambda<0,\,\,\,\,\,\,\,\,\,\,\,\,\,\,\text{if}\,\,\,N=3,4,\\{}\\
\lambda<4-N,\,\,\text{if}\,\,N\geq 5.
\end{array}\right.$$
Then, choosing $$0<j<\min\biggl\{\frac{\lambda-N}{\lambda}, \frac{\lambda-N}{2(\lambda-1)}, \frac{4-N-\lambda}{2}\biggr\}$$ we verify that $$\frac{\lambda j- \lambda+N}{2}<0,\,\,\frac{2(\lambda-1) j- \lambda+N}{2} <0\,\,\,\text{and}\,\,\,\frac{2j+ \lambda+N-4}{2}<0.$$
Since $$\frac{\lambda-N}{\lambda}>0,\,\,\frac{\lambda-N}{2(\lambda-1)}>0\,\,\,\,\text{and} \,\,\,\,\frac{4-N-\lambda}{2}>0,$$ the set $\left(0,\min\biggl\{\frac{\lambda-N}{\lambda}, \frac{\lambda-N}{2(\lambda-1)}, \frac{4-N-\lambda}{2}\biggr\}\right)$ is
 also not empty.

Therefore, passing to the limit as $R\to+\infty$ in \eqref{++}, a contradiction follows $A>0$, which completes the proof.
\end{proof}

\section*{Declaration of competing interest}
The authors declare that there is no conflict of interest.

\section*{Data Availability Statements} The manuscript has no associated data.

\section*{Acknowledgments}
This research has been funded by the Science Committee of the Ministry of Education and Science of the Republic of Kazakhstan (Grant No. AP14869090). Berikbol Torebek is also supported by the FWO Odysseus 1 grant G.0H94.18N: Analysis and Partial Differential Equations and by the Methusalem programme of the Ghent University Special Research Fund (BOF) (Grant number 01M01021).


\begin{thebibliography}{99}
\bibitem{Bandle} C. Bandle, H. A. Levine, On the existence and nonexistence of global solutions of reaction-diffusion equations in sectorial domains, {\it Trans. Am. Math. Soc.} 316(2) (1989), 595--622.
\bibitem{Bandle2} C. Bandle, H. A. Levine, Q. S. Zhang, Critical exponents of Fujita type for inhomogeneous parabolic equations and systems, {\it J. Math. Anal. Appl.} 251 (2000), 624--648.
\bibitem{Torebek} M. Borikhanov, B. T. Torebek, Nonexistence of global solutions for an inhomogeneous pseudo-parabolic equation, {\it Appl. Math. Lett.} 134 (2022), 108366.
\bibitem{Dabbicco} M. D'Abbicco, R. Ikehata, H. Takeda,
Critical exponent for semi-linear wave equations with double damping terms in exterior domains, {\it Nonlinear Differ. Equ. Appl.}, 26:6 (2019), 56.
\bibitem{Fino} A. Z. Fino, H. Ibrahim, A. Wehbe, A blow-up result for a nonlinear damped wave equation in exterior domain: the critical case, {\it Comput. Math. Appl.}, 73 (2017), 2415--2420.
\bibitem{Fujita} H. Fujita. On the blowing up of solutions of the Cauchy problem for $u_t=\triangle u+u^{1+\alpha}$. {\it J. Fac. Sci. Univ. Tokyo Sect.} 13 (1966), 109--124.
\bibitem{Ikeda} M. Ikeda, M. Jleli, B. Samet, On the existence and nonexistence of global solutions for certain semilinear exterior problems with nontrivial Robin boundary conditions, {\it J. Differential Equations}, 269:1 (2020), 563--594.
\bibitem{Ikeda1} M. Ikeda, M. Sobajima, Remark on upper bound for lifespan of solutions to semilinear evolution equations in a two-dimensional exterior domain, {\it J. Math. Anal. Appl.}, 470 (2019), 318--326.
\bibitem{Ikehata} R. Ikehata, Y. Inoue, Global existence of weak solutions for two-dimensional semilinear wave equations with strong damping in an exterior domain, {\it Nonlinear Anal.}, 68:1 (2008), 154--169.
\bibitem{Vetro} M. Jleli, B. Samet, C. Vetro, A blow-up result for a nonlinear wave equation on manifolds: the critical case, {\it Appl. Anal.}, 102:5 (2023), 1463--1472.
\bibitem{Jleli2} M. Jleli, B. Samet, D. Ye, Critical criteria of Fujita type for a system of inhomogeneous wave inequalities in exterior domains. {\it J. Differential Equations}, 268:6 (2020), 3035--3056.
\bibitem{Jleli0} M. Jleli, M. Kirane, B. Samet, Blow-up results for higher-order evolution differential inequalities in exterior domains. {\it Adv. Nonlinear Stud.}, 19 (2019), 375--390.
\bibitem{Jleli} M. Jleli, B. Samet, New blow-up results for nonlinear boundary value problems in exterior domains, {\it Nonlinear Anal.} 178 (2019), 348--365.
\bibitem{Jleli1} M. Jleli, B. Samet, C. Vetro, On the critical behavior for inhomogeneous wave inequalities with Hardy potential in an exterior domain, {\it Adv. Nonlinear Anal.}, 10:1 (2021), 1267--1283.
\bibitem{Laptev1} G. G. Laptev, Nonexistence results for higher-order evolution partial differential inequalities, {\it Proc. Amer. Math. Soc.}, 131:2 (2003), 415--423.
\bibitem{Levine} H. A. Levine, Q. S. Zhang, The critical Fujita number for a semilinear heat equation in exterior domains with homogeneous Neumann boundary values, {\it Proc. R. Soc. Edinb. A}, 130 (2000), 591--602.
\bibitem{Ogawa} T. Ogawa, H. Takeda, Non-existence of weak solutions to nonlinear damped wave equations in exterior domains, {\it Nonlinear Anal.}, 70 (2009), 3696--3701.
\bibitem{Pinsky} R. Pinsky, The Fujita exponent for semilinear heat equations with quadratically decaying potential or in an exterior domain. {\it J. Differential Equations}, 246:6 (2009), 2561--2576.
\bibitem{Rault} J.-F. Rault, The Fujita phenomenon in exterior domains under the Robin boundary conditions, {\it C. R. Math. Acad. Sci. Paris}, 349 (2011), 1059--1061.
\bibitem{Sun} Y. Sun, The absence of global positive solutions to semilinear parabolic differential inequalities in exterior domain, {\it Proc. Amer. Math. Soc.}, 145:8 (2017), 3455--3464.
\bibitem{Zhang} Q. S. Zhang, A general blow-up result on nonlinear boundary-value problems on exterior domains, {\it Proc. Roy. Soc. Edinburgh Sect. A} 131:2 (2001), 451--475.
\bibitem{Zeng} X. Zeng, Existence and nonexistence of global positive solutions for the evolution p-Laplacian equations in exterior domains, {\it Nonlinear Anal.}, 67 (2007), 901--916.
\bibitem{Zeng1} X. Zeng, Z. Liu, Existence and nonexistence of global positive solutions for degenerate parabolic equations in exterior domains, {\it Acta Math. Sci.}, 30 (2010), 713--725.
\end{thebibliography}
\end{document}